\documentclass[12pt,reqno]{amsart}

\usepackage{amsmath,amsfonts,amssymb,amscd,amsthm,calc,enumerate}
\usepackage[english]{babel}
\usepackage{yfonts,color}

\newtheorem{theorem}{Theorem}[section]

\newtheorem{lemma}[theorem]{Lemma}
\newtheorem{corollary}[theorem]{Corollary}
\newtheorem{proposition}[theorem]{Proposition}

\newtheorem{question}[theorem]{Question}

\theoremstyle{definition}
\newtheorem{definition}[theorem]{\bf Definition}
\newcommand{\darrow}{\!\downarrow}
\newcommand{\uarrow}{\!\uparrow}
\newcommand{\la}{\langle}
\newcommand{\ra}{\rangle}

\renewcommand{\leq}{\leqslant}
\renewcommand{\geq}{\geqslant}

\newcommand{\vph}{\varphi}

\newcommand{\A}{\mathcal{A}}
\newcommand{\K}{\mathcal{K}}
\newcommand{\U}{{\sf{U}}}

\newcommand{\PA}{\mathrm{PA}}

\begin{document}

\title[Fixed point theorems for precomplete numberings]
{Fixed point theorems for precomplete numberings}

\author[H. P. Barendregt]{Henk Barendregt}
\address[Henk Barendregt]{Radboud University Nijmegen\\
Institute for Computing and Information Sciences\\
P.O. Box 9010, 6500 GL Nijmegen, the Netherlands.
} \email{henk@cs.ru.nl}

\author[S. A. Terwijn]{Sebastiaan A. Terwijn}
\address[Sebastiaan A. Terwijn]{Radboud University Nijmegen\\
Department of Mathematics\\
P.O. Box 9010, 6500 GL Nijmegen, the Netherlands.
} \email{terwijn@math.ru.nl}

\begin{abstract}
In the context of his theory of numberings, Ershov showed
that Kleene's recursion theorem holds for any precomplete numbering. 
We discuss various generalizations of this result. 
Among other things, we show that 
Arslanov's completeness criterion also holds for every 
precomplete numbering, and we discuss the relation with 
Visser's ADN theorem, as well as the uniformity or nonuniformity 
of the various fixed point theorems. 
Finally, we base numberings on partial combinatory algebras
and prove a generalization of Ershov's theorem in this context. 
\end{abstract}



\date{\today}

\maketitle

\section{Introduction}

In this paper we discuss various fixed point theorems in 
computability theory, and related areas such as 
$\lambda$-calculus and combinatory algebra. 
The starting point is Kleene's famous recursion theorem~\cite{Kleene}, 
which was generalized to precomplete numberings by Ershov~\cite{Ershov2}. 
These are discussed in section~\ref{recthm}, after we first discuss 
Ershov's theory of numberings in section~\ref{numberings}.

The recursion theorem was generalized in other ways by 
Visser~\cite{Visser} and Arslanov~\cite{Arslanov}. 
Visser proved the so-called `anti diagonal normalization theorem'
that we discuss in section~\ref{sec:ADN}.
Arslanov extended the recursion theorem from computable functions 
to arbitrary functions computable from an incomplete c.e.\ Turing degree. 
The Arslanov completeness criterion states that 
a c.e.\ set is Turing complete if and only if it 
computes a fixed point free function. 
Recently, a joint generalization of Arslanov's completeness criterion
and the ADN theorem was given by Terwijn~\cite{Terwijna}.
We discuss Arslanov's completeness criterion in section~\ref{sec:Arslanov}.

Finally, in sections~\ref{sec:pca} and \ref{sec:completeness}, we
discuss the relation with Fef\-er\-man's version of the recursion theorem
for partial combinatory algebras \mbox{(pca's)} \cite{Feferman}.  Here we
base the notion of numbering on pca's of arbitrary cardinality, and
prove a fixed point theorem for these (Theorem~\ref{recthmpca}). This
generalizes Ershov's recursion theorem in this setting.

Our notation from computability theory is mostly standard.
In the following, $\vph_n$ denotes the $n$-th partial computable
(p.c.) function, in some standard numbering of the p.c.\ functions.
Partial computable (p.c.) functions are denoted by lower case
Greek letters, and (total) computable functions by lower case
Roman letters.
$\omega$ denotes the natural numbers.
The set $W_e$ denotes the domain of the p.c.\ function $\vph_e$.
We write $\vph_e(n)\darrow$ if this computation is defined,
and $\vph_e(n)\uarrow$ otherwise.
We let $\la e,n\ra$ denote a computable pairing function.
$\emptyset'$ denotes the halting set.
For unexplained notions we refer to Odifreddi~\cite{Odifreddi} or
Soare~\cite{Soare}.

\section{Numberings and equivalence relations}\label{numberings}

The theory of numberings (also called numerations, after the German
`Numerierung') was initiated by Ershov.  The following concepts were
introduced by him in \cite{Ershov}.

\begin{definition} \label{basic}
A {\em numbering\/} of a set $S$ is a surjection
$\gamma\colon\omega\rightarrow S$. 
Given $\gamma$, define an equivalence relation on $\omega$ by 
$n \sim_\gamma m$ if $\gamma(n) = \gamma(m)$. 

A numbering $\gamma$ is {\em precomplete\/} if for every partial
computable unary function $\psi$ there exists a computable unary $f$
such that for every~$n$.
\begin{equation} \label{precomplete}
\psi(n)\darrow \; \Longrightarrow \; f(n) \sim_\gamma \psi(n). 
\end{equation}
Following Visser, we say that 
{\em $f$ totalizes $\psi$ modulo~$\sim_\gamma$\/} if \eqref{precomplete} holds.

A precomplete numbering $\gamma$ is {\em complete\/} if there is a
\emph{special element} $a{\in}\omega$ such that next to
\eqref{precomplete} also $f(n) \sim_\gamma a$ for every $n$ with
$\psi(n)\uarrow$.
\end{definition}

The prime example of a numbering is $n \mapsto \vph_n$ for the set of
unary p.c.\ functions.  This numbering is precomplete: by the
S-m-n-theorem, for any p.c.\ $\psi$ there is a (total) computable $f$
such that $\vph_{f(n)} = \vph_{\psi(n)}$ for every~$n$ such that
$\psi(n)\darrow$.  The numbering is even complete: as required special
element we can take the totally undefined function.

The numbering $n\mapsto W_n$ of the c.e.\ sets is closely 
related (and for our purposes below equivalent) to the numbering 
of the p.c.\ functions. 
It is also complete, with as special element the empty set. 

Other examples of numberings come from $\lambda$-calculus.  For
example, the closed $\lambda$-terms, modulo $\beta$-equality, can be
enumerated as a precomplete numbering\footnote{By $\gamma(n)={\sf
    E}{\bf c}_n$, where ${\sf E}$ is a $\lambda$-term enumerating
  closed terms and ${\bf c}_n$ is the $n$-th numeral adequately
  representing natural numbers in $\lambda$-calculus.},
cf.\ Visser~\cite[p261,264]{Visser}, referring to Barendregt. If
moreover unsolvable $\lambda$-terms are equated, then this numbering
even becomes complete. Other examples can be found in \cite{Visser},
and still more examples come from pca's, that we discuss in
section~\ref{sec:pca} below.

Numberings and equivalence relations are mutually 
related~\cite{BernardiSorbi}. 
For every numbering $\gamma$ we have the corresponding 
equivalence $\sim_\gamma$. 
Conversely, given an equivalence relation $R$ on $\omega$ 
(or any other countable set), 
we have the numbering $n\mapsto [n]$ of the equivalence 
classes of~$R$.
Hence the study of numberings is equivalent to that of 
(countable) equivalence relations. In particular we can also 
apply the terminology of Definition~\ref{basic} to such relations, 
and talk about precomplete and complete equivalence relations. 

A class of countable equivalence relations that is of particular interest 
is the class of computably enumerable equivalence relations, 
simply called {\em ceers}. 
These were studied by Ershov~\cite{Ershov1977} in the context 
of the theory of numberings (though examples of them occurred earlier 
in the literature), 
and in early writings were called {\em positive\/} equivalence relations. 
Bernardi and Sorbi~\cite{BernardiSorbi} proved that every precomplete ceer
is m-com\-plete (even with an extra uniformity condition). They also showed 
that this implies 1-completeness \cite[p532]{BernardiSorbi}. 
This result was later strengthened by Lachlan~\cite{Lachlan} 
(see also \cite[p425]{AndrewsBadaevSorbi}), who showed 
that all precomplete ceers are computably isomorphic. 
For a recent survey about ceers we refer the reader to
Andrews, Badaev, and Sorbi~\cite{AndrewsBadaevSorbi}.

An interesting example of a ceer 
(discussed in Ber\-nar\-di and Sorbi~\cite[p534]{BernardiSorbi}) 
is the Lindenbaum algebra of 
$\PA$ (Peano arithmetic). Identify formulas $\vph$ and $\psi$ in the
language of PA with their G\"odel numbers. Let $\vph \sim_\PA \psi$ if
these formulas are provably equivalent in~$\PA$.  Then $\sim_\PA$ is
obviously a ceer. This relation is not precomplete, as can be seen using
Theorem~\ref{Ershov} below: the function $\vph \mapsto \neg\vph$ is
computable, but does not have a fixed point modulo $\sim_\PA$.  By
contrast, the analogous ceer $\sim_{\Sigma_n}$, obtained by considering
the fragment of $\PA$ of $\Sigma_n$-formulas, is precomplete,
cf.\ Visser~\cite[p263]{Visser}.

\section{The recursion theorem}\label{recthm}

Kleene's recursion theorem~\cite{Kleene} states that every computable
function~$f$ has a fixed point, in the sense that 
there exists a number $n$ such that $\vph_{f(n)} = \vph_n$. 
This result holds uniformly, meaning that the fixed point can 
be found computably from a code of~$f$. 
For an extensive discussion of this fundamental theorem, 
and the many applications it has found in logic, 
see Moschovakis~\cite{Moschovakis}. 

Using Ershov's terminology, we can phrase Kleene's result 
by saying that $f$ has a fixed point modulo $\sim_\gamma$, 
where $\gamma$ is the numbering $n\mapsto\vph_n$ of the 
p.c.\ functions. 
Ershov showed that the recursion theorem holds for every 
precomplete numbering~$\gamma$ in the following way.

\begin{theorem} {\rm (Ershov's recursion theorem \cite{Ershov2})}\label{Ershov}
Let $\gamma$ be a precomplete numbering, and let $f$ be a 
computable function. Then $f$ has a fixed point modulo $\sim_\gamma$,  
i.e.\ there exists a number $n$ such that $f(n) \sim_\gamma n$. 
\end{theorem}

As is the case for Kleene's recursion theorem, this result holds uniformly. 
For later reference we explicitly state the following version: 

\begin{theorem} {\rm (Ershov's recursion theorem with parameters)}
\label{Ershovparam}
Let $\gamma$ be a precomplete numbering, and
let $h(x,n)$ be a computable binary function. 
Then there exists a computable function $f$ such that for all~$n$, 
$f(n) \sim_\gamma h(f(n),n)$.
\end{theorem}
\begin{proof}
By precompleteness, let $d$ be a computable function such that 
$$
d(x,n) \sim_\gamma \vph_x(x,n)
$$ for every $x$ and~$n$ where the latter is defined.\footnote{%
Note that we can generalize precompleteness (Definition~\ref{basic})
to functions with multiple arguments, which is allowed by the usual
coding of sequences.} Let $e$ be a code such that $\vph_e(x,n) =
h(d(x,n),n)$ for all $x$ and~$n$.  Then
$$
d(e,n) \sim_\gamma \vph_e(e,n) = h(d(e,n),n),
$$ 
so that $d(e,n)$ is a fixed point for every~$n$. 
\end{proof}

Theorem~\ref{Ershovparam} is equivalent with the following form, 
given in Andrews, Badaev, and Sorbi \cite[p423]{AndrewsBadaevSorbi}.

\begin{theorem} \label{ErshovABS}
Let $\gamma$ be a precomplete numbering.
There exists a computable function~$f$
such that for every~$n$, if $\vph_n(f(n))\darrow$ then 
$$
\vph_n(f(n)) \sim_\gamma f(n).
$$ 
\end{theorem}

Theorem~\ref{Ershovparam} and Theorem~\ref{ErshovABS} are equivalent, 
{\em for precomplete numberings}.
To see that Theorem~\ref{ErshovABS} implies Theorem~\ref{Ershovparam},
observe that, given a computable function $h$ as in the latter theorem, 
there is a computable function $g$ such that 
$\vph_{g(n)}(x) = h(x,n)$ for every $x$ and~$n$. 
For $f$ as in Theorem~\ref{ErshovABS} we then have 
$$
h(f(g(n)),n) = \vph_{g(n)}(f(g(n))) \sim_\gamma f(g(n))
$$
for every~$n$, so $f\circ g$ is the desired computable function 
producing fixed points. 

Conversely, Theorem~\ref{Ershovparam} implies Theorem~\ref{ErshovABS}.
By precompleteness of $\gamma$, there is a computable function $h$ 
that totalizes the universal p.c.\ function modulo $\sim_\gamma$, 
i.e.\ such that 
$$
\vph_n(x)\darrow \;\Longrightarrow\; h(x,n) = \vph_n(x)
$$ 
for every $x$ and~$n$. 
Now Theorem~\ref{Ershovparam} provides the required 
fixed points~$f(n)$.

The converse of Theorem~\ref{ErshovABS} also holds.
The statement of the theorem holds for a numbering $\gamma$
if and only if $\gamma$ is precomplete 
(cf. \cite[p423]{AndrewsBadaevSorbi}).
Since the equivalence of Theorem~\ref{Ershovparam} and Theorem~\ref{ErshovABS}
above uses that $\gamma$ is precomplete, it is not clear whether the 
converse of Theorem~\ref{Ershovparam} also holds. Hence we ask the following.

\begin{question}
Suppose that an arbitrary numbering $\gamma$ satisfies the 
statement of Theorem~\ref{Ershovparam}. Does it follow
that $\gamma$ is precomplete?
\end{question}

\section{The ADN theorem} \label{sec:ADN}

The ADN theorem (Theorem~\ref{ADN} below) is an extension of the 
recursion theorem, proved in Visser~\cite{Visser}.
It was motivated by developments in early proof theory, 
in particular Rosser's extension of G\"odel's incompleteness theorem. 
Visser mentions the work of Smory\'nski and Shepherdson's fixed point 
as further motivation, cf.\ \cite{Smorynski1980}. 
The analogy between the ADN theorem and Rosser's theorem was 
neatly summarized in Barendregt~\cite{Barendregt1992}
by the following mock equation.
$$
\frac{\text{G\"odel}}{\text{Rosser}} = 
\frac{\text{recursion theorem}}{\text{ADN theorem}}
$$
The analogy is further illustrated by the proof of the 
ADN theorem below. 

The ADN theorem has several interesting applications. 

\begin{enumerate}[\rm (i)]

\item[$\bullet$]
Visser himself discusses some consequences of the ADN theorem 
for the $\lambda$-calculus in~\cite{Visser}.

\item[$\bullet$]
Theorem~1 (about the m-completeness of precomplete ceers) 
in Bernardi and Sorbi~\cite{BernardiSorbi} uses $\omega+1$ 
applications of the ADN theorem. The construction in the proof 
uses the ADN theorem $\omega$ times, plus one more for Lemma~2. 

\item[$\bullet$]
Barendregt~\cite{Barendregt1992} uses the ADN theorem to prove
a result of Statman.

\item[$\bullet$]
The notion of diagonal function used in the ADN theorem 
relates nicely to the concept of fixed point free function
and similar concepts that figure prominently in 
computability theory, cf.\ the discussion below. 
 
\end{enumerate}

\begin{definition} 
A partial function $\delta$ is a {\em diagonal function\/} for the 
numbering $\gamma$ if $\delta(x) \not\sim_\gamma x$ for every 
$x$ in the domain of~$\delta$. 
\end{definition} 

N.B.\ Note that in this definition we do not require $\delta$ to be 
p.c., in contrast to the original definition in Visser~\cite{Visser}.
This is because it is also interesting to discuss the Turing degrees 
of diagonal functions in general. 

By Jockusch et al.\ \cite{Jockuschetal}, the Turing degrees of
diagonal functions for the numberings $n\mapsto \vph_n$ and $n\mapsto
W_n$ coincide.  They also coincide with the degrees of {\em diagonally
  noncomputable\/}, or DNC, functions, i.e.\ functions $g$ with $g(e)
\neq \vph_e(e)$ for every~$e$.  Diagonal functions for the numbering
$n\mapsto W_n$ of the c.e.\ sets are called {\em fixed point free\/}
(or simply FPF) in the literature.  (Usually these are total
functions, though in \cite{Terwijna} and \cite{Terwijnb} also partial
FPF functions were considered.)  DNC and FPF functions play an
important part in computability theory, for example in the work of
Ku\v{c}era~\cite{Kucera}.  See Astor~\cite{Astor} for a recent example
of their use, or Downey and Hirschfeldt~\cite{DH} for many more.  They
are also closely related to the study of complete extensions of Peano
Arithmetic, see e.g.\ the work of Jockusch and
Soare~\cite{JockuschSoare1972a}.

\begin{theorem} {\rm (ADN theorem, Visser~\cite{Visser})} \label{ADN}
Let $\gamma$ be a precomplete numbering, and 
suppose that $\delta$ is a partial computable diagonal function for~$\gamma$. 
Then for every partial computable function $\psi$ there
exists a computable function $f$ such that for every $n$,
\begin{align}
\psi(n)\darrow \; &\Longrightarrow \; f(n) \sim_\gamma \psi(n) \label{totalize} \\
\psi(n)\uarrow \; &\Longrightarrow \; \delta(f(n))\uarrow   \label{avoid}
\end{align}
\end{theorem}

\begin{definition}
Note that \eqref{totalize} expresses that 
$f$ totalizes $\psi$ modulo~$\sim_\gamma$. 
If both \eqref{totalize} and \eqref{avoid} hold, we say that 
{\em $f$ totalizes $\psi$ avoiding~$\delta$}. 
\end{definition}

Note that the ADN theorem implies Ershov's recursion theorem
(Theorem~\ref{Ershov}). Indeed, suppose towards a contradiction that
some total computable $d$ has no fixed point modulo
$\sim_\gamma$. Then $d$ is a total computable diagonal function. Then
a p.c.\ function $\psi$ with $\psi(0)\uarrow$ cannot be totalized
modulo $\sim_\gamma$ avoiding $d$ by any $f$, as we will not have
$d(f(0))\uarrow$, by the totality of $d$. This contradicts the ADN
Theorem.

\medskip\noindent
{\em Proof of Theorem~\ref{ADN}.}
We use Ershov's recursion theorem with parameters
(Theorem~\ref{Ershovparam}).
Let $\eta$ be p.c.\ such that for all $x$ and $n$, 
$$
\eta(x,n)=
\begin{cases}
\delta(x) & \text{if $\psi(n)\darrow\;<\;\delta(x)\darrow$}\footnotemark,\\
\psi(n) & \text{if $\delta(x)\darrow\;\leq\;\psi(n)\darrow$}, \\
\uparrow & \text{otherwise.}
\end{cases}
$$
\footnotetext{%
By this we mean that the first stage at which $\psi(n)$ converges 
is less than that of $\delta(x)$, if the latter converges at all.}

\noindent By precompleteness of $\gamma$, there is a computable function $h$
that totalizes $\eta$ modulo $\sim_\gamma$. 
Let $f$ be as in Ershov's recursion theorem with parameters 
(Theorem~\ref{Ershovparam}). Then for every~$n$,  
$$
f(n) \sim_\gamma h(f(n),n) \sim_\gamma \eta(f(n),n),
$$ 
whenever the latter is defined. 
Now $f(n) \sim_\gamma \delta(f(n))\darrow$ is impossible, 
since $\delta$ is a diagonal for $\gamma$, and hence 
$f$ totalizes $\psi$ avoiding~$\delta$.
\qed

\medskip
By taking $\psi$ in Theorem~\ref{ADN} universal, 
we see that the following uniform version holds.

\begin{theorem} {\rm (ADN theorem, uniform version)} \label{ADNuniform}
Let $\gamma$ be a precomplete numbering, and 
suppose that $\delta$ is a partial computable diagonal function for~$\gamma$. 
Then there exists a computable function $f$
such that for every fixed $e$ the function 
$f(\la e,n\ra)$ totalizes $\vph_e$ avoiding~$\delta$. 
\end{theorem}
\begin{proof}
Consider the universal function $\psi(\la e,n\ra) = \vph_e(n)$. 
By Theorem~\ref{ADN}, there exists a computable $f$ that 
totalizes $\psi$ avoiding $\delta$. Hence
\begin{align*}
\psi(\la e,n\ra)=\vph_e(n)\darrow \; &\Longrightarrow \; f(\la e,n\ra) \sim_\gamma \vph_e(n)\\
\psi(\la e,n\ra)=\vph_e(n)\uarrow \; &\Longrightarrow \; \delta(f(\la e,n\ra))\uarrow, 
\end{align*}
and therefore $f(\la e,n\ra)$ totalizes $\vph_e$ avoiding~$\delta$.
\end{proof}

Theorem~\ref{ADNuniform} shows that Theorem~\ref{ADN} is uniform 
in a code of~$\psi$. 
A careful reading of the proof of the ADN theorem above shows that 
it is also uniform in a code $d$ of~$\delta$.\footnote{%
This means that there is a computable function $f = f(d,n)$ 
such that, if $\delta = \vph_d$ is a diagonal function 
for $\gamma$, then $f(d,n)$ totalizes $\psi$ avoiding~$\delta$.}
It is shown in Terwijn~\cite{Terwijnb} that 
(for the numbering $n\mapsto W_n$) 
neither Arslanov's completeness criterion nor the ADN theorem
have a version with parameters analogous to 
the recursion theorem with parameters. 
(Note that the ADN theorem is in a way a contrapositive 
formulation of the recursion theorem, so that some care is 
needed in how to formulate the parameterized version.) 
A fortiori, the same holds in the context of arbitrary 
precomplete numberings. 

\section{Arslanov's completeness 
criterion for precomplete numberings}\label{sec:Arslanov}

Arslanov's completeness criterion \cite{Arslanov} states that 
a c.e.\ set $A$ is Turing complete if and only if 
$A$ can compute a FPF function, i.e.\ a function $f$ such that 
$W_{f(n)}\neq W_n$ for every~$n$.
Note that this vastly extends Kleene's recursion theorem, namely 
from computable sets to incomplete c.e.\ sets. 
The condition that $A$ is c.e.\ is necessary, as by the 
low basis theorem~\cite{JockuschSoare1972b} there exist FPF functions of 
low Turing degree.

In the next theorem we formulate Arslanov's completeness criterion 
for arbitrary precomplete numberings. The usual version of 
the completeness criterion corresponds to the case where 
$\gamma$ is the standard numbering of the c.e.\ sets $n\mapsto W_n$. 
(Or equivalently, by the aforementioned result of 
Jockusch et al.\ \cite{Jockuschetal}, 
the numbering of the p.c.\ functions $n\mapsto \vph_n$.)

\begin{theorem} \label{Arslanov} 
Suppose $\gamma$ is a precomplete numbering, and 
$A <_T \emptyset'$ is an incomplete c.e.\ set. 
If $g$ is an $A$-computable function, then $g$ has a fixed point modulo $\gamma$, 
i.e.\ there exists $n{\in}\omega$ such that $g(n) \sim_\gamma n$.
\end{theorem}
\begin{proof}
The following proof is a modification of the proof in Soare~\cite{Soare}.

Since $g\leq_T \emptyset'$, by Shoenfield's limit lemma~\cite{Shoenfield}
there is a computable approximation $\hat g$ such that 
$$
g(n) = \lim_{s\rightarrow\infty} \hat g(n,s)
$$
for every~$n$. 
Because $A$ is c.e., this approximation has a modulus $m\leq_T A$, that is, 
for all $s\geq m(n)$ we have $g(n) = \hat g(n,s)$. 
Now let $\eta$ be partial computable such that
$$
\eta(x,n) =
\begin{cases}
\hat g(x,s_n) & \text{if 
$s_n$ is the least number $s$ such that $n{\in}\emptyset'_s$}, \\
\uparrow & \text{if such $s$ does not exist.}
\end{cases}
$$
By the precompleteness of $\gamma$, let $h$ be a computable function that 
totalizes $\eta$ modulo~$\sim_\gamma$, so that 
$h(x,n) \sim_\gamma \eta(x,n)$ whenever the latter is defined.
By Ershov's recursion theorem with parameters (Theorem~\ref{Ershovparam}), 
let $f$ be a computable function such that 
$f(n)\sim_\gamma h(f(n),n)$ for every~$n$. 
In particular we have 
$f(n)\sim_\gamma \eta(f(n),n) \sim_\gamma \hat g(f(n),s_n)$ when $n{\in}\emptyset'$. 

We claim that there exists $n{\in} \emptyset'$ such that 
$\hat g(f(n),s_n) = g(f(n))$, so that $f(n)$ is a fixed point of~$g$. 
Otherwise we would have that for every~$n$, if $n{\in}\emptyset'$ then 
$\hat g(f(n),s_n) \neq g(f(n))$, and hence $m(f(n))> s_n$. 
It follows that $n{\in}\emptyset' \Leftrightarrow n{\in}\emptyset'_{m(f(n))}$, 
and hence $\emptyset'\leq_T A$, contrary to assumption.
\end{proof}

A joint generalization of the ADN theorem and 
Arslanov's completeness criterion 
for the numbering $n\mapsto W_n$ of the c.e.\ sets 
was given in Terwijn~\cite{Terwijna}.
At this point it is not clear that the proof in \cite{Terwijna} 
generalizes to arbitrary precomplete numberings, though we 
conjecture that it is possible to adapt the proof. 

\begin{question}
Does the joint generalization Theorem~5.1 in \cite{Terwijna} 
hold for arbitrary precomplete numberings?
\end{question}

\section{Numberings and partial combinatory algebra}\label{sec:pca}

In this section we discuss the relation of the theory of 
numberings with partial combinatory algebra. 
Combinatory algebra was introduced by Sch\"onfinkel~\cite{Schoenfinkel}, 
and partial combinatory algebra in Feferman~\cite{Feferman}.
We begin by repeating some relevant definitions. A fuller 
account of partial combinatory algebra can be found in 
van Oosten~\cite{Oosten}, from which we also borrow some of 
the terminology. 

A {\em partial applicative structure\/} (pas) is a set $\A$ together
with a partial map from $\A\times \A$ to $\A$.  We denote the image of
$(a,b)$, if it is defined, by $ab$, and think of this as `$a$ applied
to $b$'.  If this is defined we denote this by $ab\darrow$.  By
convention, application associates to the left. We write $abc$ instead
of $(ab)c$.  {\em Terms\/} over $\A$ are built from elements of $\A$,
variables, and application. If $t_1$ and $t_2$ are terms then so is
$t_1t_2$.  If $t(x_1,\ldots,x_n)$ is a term with variables $x_i$, and
$a_1,\ldots,a_n {\in} \A$, then $t(a_1,\ldots,a_n)$ is the term obtained
by substituting the $a_i$ for the~$x_i$.  For closed terms
(i.e.\ terms without variables) $t$ and $s$, we write $t \simeq s$ if
either both are undefined, or both are defined and equal. Here
application is \emph{strict} in the sense that for $t_1t_2$ to be defined,
it is necessary (but not sufficient) that both $t_1,t_2$ are defined.
We say that an element $f{\in} \A$ is {\em total\/} if $fa\darrow$ for 
every $a{\in} \A$.

\begin{definition} \label{combinatorycomplete}
A pas $\A$ is {\em combinatory complete\/} if for any term
$t(x_1,\ldots,x_n,x)$, $0\leq n$, with free variables among
$x_1,\ldots,x_n,x$, there exists a $b{\in} \A$ such that
for all $a_1,\ldots,a_n,a{\in} \A$,
\begin{enumerate}[\rm (i)]

\item $ba_1\cdots a_n\darrow$,

\item $ba_1\cdots a_n a \simeq t(a_1,\ldots,a_n,a)$.

\end{enumerate}
A pas $\A$ is a {\em partial combinatory algebra\/} (pca) if 
it is combinatory complete. 
\end{definition}

The property of combinatory completeness allows for the following 
definition in any pca. 
For every term $t(x_1,\ldots,x_n,x)$, $0\leq n$, with free variables among
$x_1,\ldots,x_n,x$, one can explicitly define
a term $\lambda x.t$ with variables among $x_1,\ldots,x_n$,
with the property that for all $a_1,\ldots,a_n,a {\in} \A$,
\begin{enumerate}[\rm (i)]

\item $(\lambda x.t)(a_1,\ldots, a_n)\darrow$,

\item $(\lambda x.t)(a_1,\ldots, a_n)a \simeq t(a_1,\ldots,a_n,a)$.

\end{enumerate}
This is noted in Feferman~\cite[p95]{Feferman}, 
and makes use of the Curry combinators $k$ and~$s$ familiar from 
combinatory logic. In fact, for any pas, the existence of such combinators 
is equivalent to being a pca \cite[p3]{Oosten}.

The prime example of a pca is {\em Kleene's first model\/} $\K_1$, 
consisting of $\omega$, with application defined by $nm = \vph_n(m)$. 
This structure is combinatory complete by the S-m-n-theorem. 
However, there are many other examples, including uncountable 
structures, see Section~1.4 of~\cite{Oosten}.

Another important example of a pca is 
{\em Kleene's second model\/} $K_2$ \cite{KleeneVesley}. 
This is a pca defined on Baire space $\omega^\omega$ (often informally
referred to as the `reals'), with application defined by coding
partial continuous functionals by reals.  The application
$\alpha\beta$ is then the result of applying the functional with code
$\alpha$ to the real~$\beta$.  We will use this model below in
section~\ref{sec:completeness}.  It also plays an important role in
the theory of realizability and higher-order computability.  For a
more elaborate discussion see for example Longley and
Normann~\cite{LongleyNormann}.  

The structures $\K_1$ and $\K_2$ can also be considered as total
combinatory algebras if one restricts them to combinators
corresponding to $\lambda{\sf I}$-calculus, in which the formation of
$\lambda x.M$ only is allowed if $x$ is a free variable of $M$, see
Barendregt~\cite[Exercises 9.5.13-14]{Barendregt}.

We note the following about the pca $\K_1$.

\begin{enumerate}[\rm (i)]

\item[$\bullet$]
The notion of precompleteness \eqref{precomplete}
generalizes the property that one can totalize any p.c.\ 
function $\psi$ on codes. 
This property gives Ershov's form of the recursion theorem
(Theorem~\ref{Ershov}).

\item[$\bullet$]
Pca's generalize the applicative structure of $nm = \vph_n(m)$.
The property of combinatory completeness may be seen as 
an abstraction 
of Kleene's S-m-n-theorem.
This property {\em also\/} gives rise to a fixed point theorem
(due to Feferman), see Theorem~\ref{Feferman} below.

\end{enumerate} 

These two generalizations of properties of Kleene's model 
are more or less orthogonal. 
For numberings, there is no notion of application, 
and pca's need not be countable.

The following is Feferman's form of the recursion theorem in pca's,
inspired by the fixed point theorem in combinatory logic.

\begin{theorem} {\rm (Feferman's recursion theorem \cite{Feferman})}
\label{Feferman}
Let $\A$ be a pca. Then there exists $f{\in} \A$ such that 
for all $g{\in} \A$ 
$$
g(fg) \simeq fg.
$$
\end{theorem}

Comparing this to Ershov's recursion theorem 
(in the form of Theorem~\ref{ErshovABS}), 
we see that Feferman's version is more general in that it applies 
to arbitrary pca's, but that it is also weaker in that the 
`function' $f$ giving the fixed point $fg$ does 
not have to be total.\footnote{%
There is a second version of the recursion theorem for pca's 
in \cite{Oosten}, namely that there is a term $f{\in} \A$ such 
that $fg\darrow$ for every $g$, and such that 
$g(fg)a \simeq fga$ for every $a{\in} \A$.
Since $f$ is total, this version does imply 
Theorem~\ref{ErshovABS}, but only for the special case of 
the numbering $n\mapsto \vph_n$.}
In some cases $f$ may be total (as for example in 
Theorem~\ref{ErshovABS}, or in the case that $\A$ is a combinatory 
algebra, i.e.\ a pca in which application is total), 
but in general $f$ cannot be total. This is obviously the case 
when the pca has a totally undefined element~$g$.\footnote{%
This is in fact the case in every nontotal combinatory algebra.
As soon as there is one application $ab$ that is undefined, $\A$ has a 
totally undefined element, namely $f=\lambda x.ab = s(ka)(kb)$.
In this case $f$ clearly satisfies Theorem~\ref{Feferman}.
Although the theorem is thus quite weak as an extension from 
combinatory algebra, its use for us is that it suggests the 
generalization of Ershov's recursion theorem to pca's that we prove 
below (Theorem~\ref{recthmpca}). 

An alternative formulation of the recursion theorem in pca's, 
analogous to Theorem~\ref{ErshovABS}, would be: 
There exists a total $f{\in}\A$ such that for all $g{\in}\A$, 
if $g(fg)\darrow$ then $g(fg)\sim fg$. This, however, does not 
hold in general by Proposition~\ref{prop2} below.} 
We will comment further on this at the end of 
section~\ref{sec:completeness}. 

We now proceed by showing how a combination of the fixed point 
theorems of Ershov and Feferman can be obtained.
We extend the notions of numbering and precompleteness 
of numberings from $\omega$ to arbitrary pca's as follows.

\begin{definition} \label{generalized}
Suppose that $\A$ is a pca, $S$ is a set, and  
$\gamma\colon \A \rightarrow S$ is surjective. 
We call $\gamma$ a {\em (generalized) numbering}. 
Define an equivalence relation on $\A$ by 
$a \sim_\gamma b$ if $\gamma(a) = \gamma(b)$. 

Call $\gamma$ {\em precomplete\/} if for every term $t(x)$ 
with one variable $x$, 
there exists a total element $f{\in} \A$ such that 
\begin{equation} \label{precomplete2}
t(a)\darrow \; \Longrightarrow \; fa \sim_\gamma t(a) 
\end{equation}
for every~$a{\in} \A$.
In this case, we say that {\em $f$ totalizes $t$ modulo~$\sim_\gamma$\/}.

As before, we say that a generalized precomplete numbering $\gamma$ is 
{\em complete\/} if there is a special element $c{\in}\A$ such that 
in addition to \eqref{precomplete2}, 
$fa \sim_\gamma c$ for every $a$ with $t(a)\uarrow$.
\end{definition}

\begin{lemma}\label{altdef}
Let $\A$ be a pca, and let $\gamma\colon \A \rightarrow S$ be a generalized 
numbering. Then the following are equivalent. 
\begin{enumerate} \item[{\rm (i)}] $\gamma$ is precomplete. 
\item[{\rm (ii)}] For every $b{\in} \A$ 
there exists a total element $f{\in} \A$ such that 
for all $a{\in} \A$, 
\begin{equation*} 
b{a}\darrow \; \Longrightarrow \; f{a} \sim_\gamma b{a}. 
\end{equation*}
\item[{\rm (iii)}]
For every $b{\in} \A$ there exists a total element $f{\in} \A$ such that 
for all $n{\in}\omega$ and  $\vec{a}=a_1,\ldots,a_n{\in} \A$, 
\begin{equation*} 
b\vec{a}\darrow \; \Longrightarrow \; f\vec{a} \sim_\gamma b\vec{a}. 
\end{equation*}
\end{enumerate}
\end{lemma}
\begin{proof} (i) $\Rightarrow$ (ii).
Apply (i) to the term $b{x}$.  

(ii) $\Rightarrow$ (iii). 
With the use of the $\lambda$-terms defined above for any pca,
$n$-tuples $a_1,\ldots,a_n$ can be coded as a single element 
$\langle a_1,\ldots, a_n\rangle=\lambda z.za_1\ldots a_k$, 
from which each $a_i$ can be decoded. Indeed, for 
$\U^n_i=\lambda u_1\ldots u_n.u_i$ we have
$$
\langle a_1,\ldots, a_n\rangle \U^n_i=a_i.
$$
Now given $n, b$ define $b'=\lambda z.b(z\U^n_1)\cdots(z\U^n_n)$.
Let $f'$ totalize $b'$ modulo $\gamma$. Then
$f=\lambda x_1\ldots x_n.f'\langle x_1,\ldots,x_n\rangle$ totalizes~$b$:
if $ba_1\cdots a_n\darrow$, then
\begin{align*}
fa_1\cdots a_n &= f'\langle a_1,\ldots,a_n\rangle \\
&\sim_\gamma b'\langle a_1,\ldots,a_n\rangle \\
&= b(\langle a_1,\ldots,a_n\rangle\U^n_1)\cdots(\langle a_1,\ldots,a_n\rangle\U^n_n) \\
&= ba_1\cdots a_n.
\end{align*}

(iii) $\Rightarrow$ (i).
Given term $t({x})$, apply (ii) with $n{=}1$ to
$b{=} \lambda {x}. t({x})$.
\end{proof}

By Lemma~\ref{altdef}, the notion of precompleteness from 
Definition~\ref{basic} is a special case of Definition~\ref{generalized}, 
namely the case where $\A$ is the pca $\K_1$, 
with application $nm = \vph_n(m)$. 
Hence we see that Ershov's recursion theorem (Theorem~\ref{ErshovABS}) 
is a special case of the following theorem.

\begin{theorem} \label{recthmpca}
Suppose $\A$ is a pca, and that 
$\gamma\colon \A \rightarrow S$ is a precomplete numbering.  
Then there exists a total $f{\in} \A$ such that for all $g{\in} \A$,  
if $g(fg)\darrow$ then 
$$
g(fg) \sim_\gamma fg.
$$
\end{theorem}
\begin{proof} The proof mimics 
$\Theta=(\lambda xy.y(xxy))(\lambda xy.y(xxy))$, 
the fixed point operator of Turing \cite{Turing}.
Let $t(x,y) = y(xxy)$. By Lemma \ref{altdef}
there is a 
$u{\in} \A$ that totalizes the term $t(x,y)$ modulo $\sim_\gamma$.
Then $uab\darrow$, for all $a,b{\in} \A$, and 
$b(aab)\darrow$ implies $uab \sim_\gamma b(aab)$.
Take $f = uu$. Then $f$ is total, because $uua\darrow$ for every $a{\in} \A$. 
Suppose for a $g{\in} \A$ one has $g(fg)\darrow$. 
Then $g(uug)\darrow$ and
\begin{align*}
fg &= uug \\
   &\sim_\gamma g(uug) \\
   &= g(fg).
\qedhere
\end{align*}
\end{proof}

\section{Combinatory completeness and precompleteness}
\label{sec:completeness}

With every pca $\A$ we have an associated generalized numbering 
$\gamma_\A\colon \A \rightarrow \A$, which is just the identity. 
We will discuss the relation between combinatory completeness of $\A$
and the precompleteness of~$\gamma_\A$. 

Combinatory completeness is the property in pca's analogous 
to the S-m-n-theorem, and precompleteness of (generalized) 
numberings (Definition~\ref{generalized}) generalizes the 
property that every p.c.\ function can be totalized modulo 
equivalence of codes, i.e.\ that the numbering 
$n\mapsto \vph_n$ is precomplete. 
Now the latter fact is proved using the S-m-n-theorem, 
so one might think that perhaps the property of 
combinatory completeness of a pca $\A$ {\em implies\/} that of 
precompleteness of the associated numbering~$\gamma_\A$.  
We now show that this is not the case, and hence that the 
assumption of precompleteness in Theorem~\ref{recthmpca}
is not superfluous. 
Recall Kleene's second model $\K_2$ from section~\ref{sec:pca}.

\begin{proposition}\label{prop}
Kleene's second model $\K_2$ is not precomplete, meaning that its 
associated generalized numbering $\gamma_{\K_2}$ is not precomplete. 
\end{proposition}
\begin{proof} 
According to Lemma~\ref{altdef}, 
we have to prove that there is a partial continuous functional
$\psi\colon\omega^\omega \rightarrow \omega^\omega$ that does not have 
a total continuous extension. For every finite string 
$\sigma{\in}\omega^{<\omega}$, denote by $[\sigma]$ the basic open
set consisting of all $X{\in}\omega^\omega$ that have $\sigma$ as 
an initial segment. 
Now define $\psi$ on every basic open $[0^n1]$ by mapping 
it continuously to 
$[0^n1]$ if $n$ is even, and to 
$[10^{n-1}]$ if $n$ is odd. 
We let $\psi$ be undefined on the rest of $\omega^\omega$. 
Then $\psi$ is continuous on its domain. 
Now consider the all zero sequence $0^\omega$, and 
suppose that $f$ is a total continuous extension of~$\psi$. 
Since the reals $0^n1 0^\omega$ converge to $0^\omega$
for $n\rightarrow\infty$, their images under~$f$ should 
converge to $f(0^\omega)$. 
But for even $n$, $f(0^n1 0^\omega)$ tends to $0^\omega$, 
and for odd $n$ it tends to $10^\omega$. 
Hence every continuous extension $f$ of $\psi$ must have both 
$f(0^\omega) = 0^\omega$ and $f(0^\omega) = 10^\omega$, 
which is impossible.
\end{proof}

\begin{corollary}
Combinatory completeness of a pca $\A$ does not imply 
precompleteness of the associated numbering $\gamma_\A$.
\end{corollary}
\begin{proof} 
As $\K_2$ is a pca, this is immediate from Proposition~\ref{prop}.
\end{proof}

We already noted that in general it is not possible to have 
the $f$ in Feferman's recursion theorem (Theorem~\ref{Feferman}) total. 
For $\K_2$, we can in fact say a bit more. 

\begin{proposition}\label{prop2}
In Kleene's second model $\K_2$, for every total element $f$
there exists a total element $g$ such that 
$g(fg) \not\simeq fg$.
\end{proposition}
\begin{proof} 
Given the code $f$ of a total continuous functional on 
$\omega^\omega$, we define a total continuous functional $g$ such 
that $g(fg) \not\simeq fg$. 

The particulars of the coding of $K_2$ are not essential 
to the proof. (The interested reader can find them in 
Longley and Normann~\cite{LongleyNormann}.)
What is needed is that if $fg\darrow$, this computation uses only 
a {\em finite\/} part of the coding of $g$ (this is precisely what 
it means for $f$ to be continuous on its domain), and further that 
the code of an element $g$ can be equal to an initial part of the 
code of the totally undefined function, and later become defined 
on a given number.  
Informally, the strategy to define $g$ is then as follows.
First let $g$ be totally undefined, until $fg$ commits to 
a certain value on $(fg)(0)$. This has to happen since $f$ is total. 
We can then diagonalize by letting the value $(g(fg))(0)$ be different 
from~$(fg)(0)$, as well as make $g$ total. 
\end{proof}

Note that Proposition~\ref{prop2} gives another proof of 
Proposition~\ref{prop}. Namely, if $\K_2$ were precomplete, 
then by Theorem~\ref{recthmpca} there would be a total 
element $f$ producing the fixed points, contradicting
Proposition~\ref{prop2}.

\medskip\noindent{\bf Acknowledgement.}
The second author thanks Jaap van Oosten for helpful 
discussions about partial combinatory algebra.

\end{document}